\documentclass[11pt,a4paper]{article}

\usepackage{euscript}
\usepackage{amsmath}
\usepackage{graphicx}
\usepackage{amsthm}
\usepackage{amssymb}
\usepackage{epsfig}
\usepackage{url}
\usepackage{colonequals}

\usepackage{yfonts}

\usepackage{amsmath,amscd}
\theoremstyle{theorem}
\newtheorem{theorem}{Theorem}[section]
\newtheorem{corollary}[theorem]{Corollary}

\newtheorem{proposition}[theorem]{Proposition}

\newtheorem{definition}{Definition}[section]

\newtheorem{example}{Example}[section]
\newtheorem{remark}{Remark}[section]
\numberwithin{equation}{section}
 
\newcommand{\s}{\sigma}

\DeclareSymbolFont{symbolsC}{U}{txsyc}{m}{n}
\DeclareMathSymbol{\notniFromTxfonts}{\mathrel}{symbolsC}{61}
\title{Prime One-Sided Ideals in Principal Ideal Domains} 
\author{Masood Aryapoor\\
	\tiny{\textit{Division of Mathematics and Physics}}\\
	\tiny{\textit{M\"{a}lardalen  University}}\\
	\tiny{\textit{Hamngatan 15, 632 17, Eskilstuna, 
			Sweden
	}}
}
 \date{}
\begin{document}
 \maketitle
\begin{abstract}
This article investigates various notions of primeness for one-sided ideals in noncommutative rings, with a focus on principal ideal domains. 

\end{abstract}
\begin{section}{Introduction} 
 The concept of primeness in ring theory has long served as a cornerstone for understanding the structural properties of rings and their ideals. While the classical notion of prime ideals in commutative rings is well-established, extending  this idea to noncommutative settings introduces a rich landscape of possibilities and challenges. 
 This article explores various generalizations of prime ideals to the realm of one-sided ideals, particularly within the framework of principal ideal domains (PIDs).
 
 In Section \ref{sec: primeness}, we begin by briefly revisiting foundational definitions such as completely prime and prime two-sided ideals, and then review four distinct notions of primeness for one-sided ideals: extremely prime, completely prime, structurally prime, and weakly prime. We examine each notion through rigorous definitions, illustrative examples, and key theorems that highlight their algebraic significance and interrelations. Section \ref{sec: primness for PID} provides several descriptions of one-sided primness for principal ideal domains. In particular, the results demonstrate how properties such as irreducibility,
 invariance, and boundedness govern the behavior of left ideals in PIDs. Our findings indicate that the notions of extremely prime, structurally prime and weakly prime ideals are primarily governed by the presence of invariant and bounded elements, whereas the concept of completely prime one-sided ideals is more closely associated with irreducibility. Although our primary focus is on PIDs, it is worth noting that some of our results hold for the more general class of 2-firs. 
 
In this article, all rings are assumed to be associative and unital, but not necessarily commutative.


\end{section}
\begin{section}{Prime One-Sided Ideals}\label{sec: primeness} 

In this section, we review several concepts of primness for both two-sided ideals and one-sided ideals. For early foundational work on these concepts, see   the article by Duan \cite{dauns1973one}. 
	\begin{subsection}{Prime ideals}
		We begin by a short discussion about the notion of primeness for two-sided ideals. Recall that a proper two-sided ideal \( \textswab{p} \) in a commutative ring \( R \) is called  prime if for \( a,b\in R\), \( ab\in \textswab{p}\) implies \( a\in \textswab{p}\) or \( b\in \textswab{p}\), or equivalently, the quotient ring \( R/  \textswab{p}\) is a domain. Generalizing this definition to  the noncommutative case leads to the following definition, which was already pointed out by H. Fitting in 1935 \cite{Fitting1935prim}:
		\begin{definition}\label{def: completely prime ideal}
			An ideal  \( \textswab{p} \) in a  ring \( R \) is said to be \emph{completely prime} if \( \textswab{p} \neq R \), and for \( a,b\in R\), \( ab\in \textswab{p}\) implies \( a\in \textswab{p}\) or \( b\in \textswab{p}\), or equivalently, the quotient ring \( R/  \textswab{p}\) is a domain.
		\end{definition}
		In  \cite{McCoy1949prime}, N. H. McCoy commented that  this notion is not useful since ``a noncommutative ring seldom contains very many completely prime ideals". For example, a simple ring with zero-divisors has no completely prime ideals. However, reduced rings (i.e., rings without nonzero nilpotent elements) have enough completely prime ideals in the sense that the intersection of all the completely prime ideals in a reduced ring is the zero ideal. More generally, completely prime ideals can be used to characterize reduced ideals. Recall that a two-sided ideal  \( \textswab{A} \) in a  ring \( R \) is called \textit{reduced} if, for \(a\in R\) and \(n\geq 1\), \( a^n\in \textswab{A}\) implies \( a\in \textswab{A}\).
		\begin{theorem}\label{thm: reduced ideal as intersection of completely prime ideals}
			An  ideal \( \textswab{A} \) of a ring \( R \) is reduced iff \( \textswab{A} \) can be written as an intersection of completely prime ideals in \( R \).
		\end{theorem}
		\begin{proof}
			The case \( \textswab{A} =\{0\} \) is due to Andrunakievich and Ryabukhin \cite{andrunakievic1968rings} (For other proofs of this result, see \cite{klein1980simple} and \cite[Theorem 12.7]{lam2001afirst}). The general case follows immediately by applying the special case to the ring \( R / \textswab{A} \). See also \cite[p. 195]{lam2001afirst}.
		\end{proof}
		We now turn to the standard definition of a prime ideal:
		\begin{definition}\label{def: prime ideal}
			An ideal  \( \textswab{p} \) in a  ring \( R \) is said to be \emph{prime} if \( \textswab{p} \neq R \), and for ideals \( \textswab{A},\textswab{B}\) in \(R\), \( \textswab{A}\textswab{B}\subseteq\textswab{p}\) implies \( \textswab{A}\subseteq\textswab{p}\) or \( \textswab{B}\subseteq\textswab{p}\).
		\end{definition}
		 In  \cite{McCoy1949prime}, N. H. McCoy initiated the systematic study of this notion and related concepts. In particular, he obtained several characterizations of prime ideals \cite[Theorem 1]{McCoy1949prime}. It is evident that every completely prime ideal is a prime ideal. Moreover, any ring has at least one prime ideal as every maximal ideal is prime.   An important result is that a two-sided ideal \( \textswab{C} \) in a  ring \( R \) is semiprime (i.e.,  \( \textswab{A}^2\subseteq\textswab{C}\) implies \( \textswab{A}\subseteq\textswab{C}\), for every ideal \(\textswab{A}\) of  \(R\)) iff \( \textswab{C} \) can be written as an intersection of prime ideals. For a detailed discussion of prime ideals and related concepts, we refer the reader to Lam's book \cite[Chapter 4]{lam2001afirst}. 
	 \begin{remark}
		We have mentioned two of the most studied notions of prime ideals in the noncommutative setting. However, other notions of prime ideals also appear in the literature; for example, the notion of a strongly prime ideal, which is motivated by the concept of pseudo-valuation rings \cite{hedstrom1978pseudo}. For another generalization of primeness, see \cite{murata1969generalization}.
	\end{remark}
	\end{subsection}

 \begin{subsection}{Extremely Prime Left Ideals}
 	The first notion of a prime one-sided ideal is obtained by generalizing the concept of a completely prime ideal:
 	\begin{definition}\label{def: extremely prime left ideal}
 		A left ideal  \( \textswab{p} \) in a  ring \( R \) is said to be \emph{extremely prime} if \( \textswab{p} \neq R \), and  and for \( a,b\in R\), \( ab\in \textswab{p}\) implies \( a\in \textswab{p}\) or \( b\in \textswab{p}\). 
 	\end{definition}
 	
 	\begin{remark}
 		Some authors use the term ``completely prime" instead of ``extremely prime". Here, we prefer the phrase ``extremely prime", suggested by  Reyes \cite{reyes2010one}, and reserve the former phrase for another notion, introduced and studied by  Reyes \cite{reyes2010one}.   
 	\end{remark}
 	A disadvantage of this concept is that there exist rings that have no completely prime one-sided ideals. Examples of such rings are simple rings with nontrivial idempotents, e.g. matrix rings over division rings, (for a proof of this, see \cite[Proposition 2.4]{reyes2010one}). It is worth noting that the condition defining the notion is left-right symmetric, which does not reflect the fact that it concerns one-sided ideals. Although the notion of an extremely prime left ideal does not seem particularly useful, we remark that in \cite{andrunakievich1984completely}, Andrunakievich utilized this notion to derive a criterion for determining whether a ring is isomorphic to a subdirect product of division rings.  

 \end{subsection}
\begin{subsection}{Completely Prime Left Ideals}
	In \cite{reyes2010one}, M.M. L.  Reyes introduced and studied a new notion of a prime one-sided ideal, which he termed ``completely prime".  He argued that completely prime one-sided ideals  control the ``one-sided structure of a ring". Let us first review the concept of a completely prime left ideal.  
	 \begin{definition}\label{def: completely prime left ideal}
		A left ideal  \( \textswab{p} \) in a  ring \( R \) is said to be \emph{completely prime} if \( \textswab{p} \neq R \), and   for \( a,b\in R\), \( ab\in \textswab{p}\) and \( \textswab{p}b\subseteq\textswab{p}\) imply \( a\in \textswab{p}\) or \( b\in \textswab{p}\). 
	\end{definition}	
	To demonstrate that the concept of a completely prime one-sided ideal is a suitable tool for studying the one-sided structure of a ring,  Reyes proved, among other things, a noncommutative version of Cohen's Theorem  (\cite[Theorem 3.8]{reyes2010one}):
	\begin{theorem}\label{thm: Noncommutative Cohen using completely prime left ideals}
		A ring \( R \) is left noetherian iff every completely prime left ideal of \( R \) is finitely generated. 
	\end{theorem}
	He also obtained a description of division rings using the notion of a completely prime left ideal (see \cite[Proposition 2.11]{reyes2010one}):
	\begin{proposition}\label{prop: division rings using completely prime left ideals}
		A nonzero ring \( R \) is a division ring iff every  proper left ideal of \( R \) is completely prime. 
	\end{proposition}
	
	Every maximal left ideal is completely prime  \cite[Corollary 2.10]{reyes2010one}. In particular, any nonzero ring has at least one completely prime left ideal. 
	It is clear that if a  left ideal is extremely prime then it is completely prime. However, a completely  prime left ideal may not be extremely problem, as shown in the following example. 
	\begin{example}\label{exam: extremely and completely}
		Let \(D\) be a division ring and \(M_2(D)\) be the ring of \(2\times 2\)-matrices over \(D\). It is easy to see that \(M_2(D)\) has no extremely prime left ideals (see also \cite[Proposition 2.4]{reyes2010one}). Note that any maximal left ideal of  \(M_2(D)\) is completely prime. 
	\end{example}
	For more results regarding completely prime one-sided ideals, see \cite{reyes2012noncommutative}. For a brief survey of notions of primeness for one-sided ideals, see also Section 8 in the reference. 
	
\end{subsection}
 
\begin{subsection}{Structurally Prime Left Ideals}
 	The notion of a prime ideal, as presented in Definition \ref{def: prime ideal}, can be naturally extended to left ideals as follows:   
 \begin{definition}\label{def: structurally prime left ideal}
 	A left ideal  \( \textswab{p} \) in a  ring \( R \) is said to be \emph{structurally prime} if \( \textswab{p} \neq R \), and  for left ideals \( \textswab{A},\textswab{B}\) in \(R\), \( \textswab{A}\textswab{B}\subseteq\textswab{p}\) implies \( \textswab{A}\subseteq\textswab{p}\) or \( \textswab{B}\subseteq\textswab{p}\).
 \end{definition}

 \begin{remark}
 	Some authors use the phrase ``prime one-sided ideals" in place of ``structurally prime one-sided ideals".  To emphasize the dependence of this notion on the structure of one-sided ideals, we adopt the term ``structurally prime". 
 \end{remark}
 It is evident that a two-sided ideal is structurally prime as a left ideal iff it is prime as a two-sided ideal. The following proposition gives several characterizations of structurally prime left ideals, which are similar to the familiar ones for prime two-sided ideals:
 \begin{proposition}\label{prop: characterization of structurally prime ideals}
 	For a left ideal \( \textswab{p} \subseteq R \), the following statements are equivalent:
 	\begin{itemize}
 		\item[(1)] \( \textswab{p} \) is structurally prime.
 		\item[(2)] For \( a,b\in R\), \(RaRb\subseteq \textswab{p}\) implies \( a\in \textswab{p} \) or \( b\in   \textswab{p} \).
 		\item[(3)] For \( a,b\in R\), \(aRb\subseteq \textswab{p}\) implies \( a\in \textswab{p} \) or \( b\in   \textswab{p} \).
 		\item[(4)]  For a two-sided ideal \( \textswab{A} \subseteq R \) and a left ideal  \( \textswab{B} \subseteq R \)  , \( \textswab{A}\textswab{B} \subseteq \textswab{p} \) implies \( \textswab{A} \subseteq \textswab{p} \) or \( \textswab{B} \subseteq \textswab{p} \).
 		\item[(5)]  For a two-sided ideal \( \textswab{A} \subseteq R \) and an element  \( b\in R \)  , \( \textswab{A}b \subseteq \textswab{p} \) implies \( \textswab{A} \subseteq \textswab{p} \) or \( b \in \textswab{p} \).
 	\end{itemize}
 \end{proposition}
 \begin{proof}
 	The proof of \((1)\implies (2)\implies (3)\implies (1)\) is straightforward. It is clear that (1) implies (4). To prove \((4)\implies (1)\), suppose  \( \textswab{A}\textswab{B}\subseteq\textswab{p}\) for some left ideals \( \textswab{A},\textswab{B}\) in \(R\). For the ideal \(\textswab{A}R \) and the left ideal \(\textswab{B}\), we have
 	\[
 		(\textswab{A}R)\textswab{B}= \textswab{A}(R\textswab{B})=\textswab{A}\textswab{B}\subseteq\textswab{p}.
 	\]
 	If (4) holds, then either  \( \textswab{A} \subseteq \textswab{A}R \subseteq \textswab{p} \) or \( \textswab{B} \subseteq \textswab{p} \), which implies (1). The easy proof of (4)\(\iff\)(5) is left to the reader.  
 \end{proof}
 	Definition \ref{def: structurally prime left ideal} appears to have been first introduced by Koh in \cite{koh1971one}, where it was used to define and study the concept of the ``right dimension of a ring." Shortly after, in \cite{michler2014prime},  Michler presented the notion and attributed it to  Lambek. Michler also proved the following noncommutative version of Cohen’s Theorem:
 	\begin{theorem}\label{thm: Noncommutative Cohen using structurally prime left ideals}
 		A ring \( R \) is left noetherian iff every structurally prime left ideal of  \( R \)  is finitely generated. 
 	\end{theorem}
 	In \cite{lambek1973torsion}, Michler and  Lambek studied  structurally prime one-sided ideals in the context of the torsion theory at a prime ideal. In particular, they 
 	proved the following result \cite[Proposition 2.1]{lambek1973torsion}: 
 	\begin{proposition}\label{prop: simplicity using structurally prime left ideals}
 		A left notherian ring \( R \) is simple iff every left ideal of \( R \) is structurally prime. 
 	\end{proposition}
	It is clear that any extremely prime left ideal is structurally prime (see the third part of Proposition \ref{prop: characterization of structurally prime ideals}). It is easy to show that any maximal left ideal is structurally prime. In particular, a structurally prime left ideal may not be extremely prime as there exist rings in which no maximal left ideal is extremely prime (see Example \ref{exam: extremely and completely}). The following example shows that a structurally prime left ideal may not be completely prime as a left ideal.   
	
	\begin{example}
		\label{exam: structurally, not completely}
		Let \(D\) be a division ring and \(M_2(D)\) be the ring of \(2\times 2\)-matrices over \(D\). The zero ideal in \(M_2(D)\) is structurally prime by Proposition \ref{prop: simplicity using structurally prime left ideals}. However, the zero ideal is not completely prime as a left ideal because \(M_2(D)\) has zero-divisors. 
	\end{example}
	The next example provides a left ideal that is completely prime, but not structurally prime. 
	\begin{example}
		\label{exam: not structurally, completely}
		Let \(D\) be a division ring. Consider the ring \(R\) of all lower triangular matrices in \(M_2(D)\). It is easy to check that 
		\[
			\textswab{p} = \left\lbrace \begin{bmatrix}
				0 & 0\\
				0 & a
			\end{bmatrix}: \, a\in D \right\rbrace 
		\] 
		is a completely prime left ideal of \(R\). The left ideal \(\textswab{p}\) is not structurally prime since
		\[
			\begin{bmatrix}
				1 & 0\\
				0 & 0
			\end{bmatrix} R \begin{bmatrix}
			0 & 0\\
			1 & 0
			\end{bmatrix} \subseteq \textswab{p},
		\]
		but 
		\[
			\begin{bmatrix}
				1 & 0\\
				0 & 0
			\end{bmatrix}\notin \textswab{p} \text{ and } \begin{bmatrix}
			0 & 0\\
			1 & 0
			\end{bmatrix}\notin \textswab{p}.
		\]
	\end{example}

\end{subsection}
\begin{subsection}{Weakly Prime Left Ideals}
	As the final notion of primeness for one-sided ideals reviewed in this paper, we consider the concept of weakly prime one-sided ideals. This notion relates to the notion of a ``structurally prime left ideal" in the same way that the concept of a ``completely prime left ideal"  relates to the notion of an ``extremely prime left ideal".   
	\begin{definition}\label{def: weakly prime left ideal}
		A left ideal  \( \textswab{p} \) in a  ring \( R \) is said to be \emph{weakly prime} if \( \textswab{p} \neq R \), and  for left ideals \( \textswab{A},\textswab{B}\) in \(R\), \( \textswab{A}\textswab{B}\subseteq\textswab{p}\) and \( \textswab{p}\textswab{B}\subseteq\textswab{p}\) imply \( \textswab{A}\subseteq\textswab{p}\) or \( \textswab{B}\subseteq\textswab{p}\).
	\end{definition}

	This notion was introduced by Koh in \cite{koh1971prime}, where the notion was termed ``prime left ideal" and the following results were obtained:
	\begin{theorem}\label{thm: Noncommutative Cohen using weakly prime left ideals}
		A ring \( R \) is left noetherian iff every weakly prime left ideal of  \( R \)  is finitely generated. 
	\end{theorem} 
	\begin{theorem}\label{thm: simplicity using weakly prime left ideals}
		A left notherian ring \( R \) is simple iff every left ideal of \( R \) is weakly prime.  
	\end{theorem} 
	It is evident that every structurally prime left ideal is weakly prime. It is easy to prove that every completely prime left ideal is weakly prime. Therefore, Theorem \ref{thm: Noncommutative Cohen using weakly prime left ideals} follows from either Theorem \ref{thm: Noncommutative Cohen using completely prime left ideals} or Theorem \ref{thm: Noncommutative Cohen using structurally prime left ideals}. Nevertheless, it is worth noting that in \cite{koh1971prime}, Koh formulated and proved his results for rings that are not necessarily unital.   
	
	In \cite[Proposition 1.1 and Proposition 1.3]{van1985weakly}, van der Walt provided several characterizations of weakly prime one-sided ideals. However, he did not observe that his notion coincides with Definition \ref{def: weakly prime left ideal}, as evidenced by the fact that he reproved one direction of Theorem \ref{thm: simplicity using weakly prime left ideals} while stating that his result strengthened Theorem \ref{thm: simplicity using weakly prime left ideals}.  Let us give a more complete version of van der Walt's characterizations of weakly prime ideals. 
	 \begin{proposition}\label{prop: characterization of weakly prime ideals}
		For a left ideal \( \textswab{p} \subseteq R \), the following statements are equivalent:
		\begin{itemize}
			\item[(1)] \( \textswab{p} \) is weakly prime.
			\item[(2)] For left ideals \( \textswab{A}, \textswab{B}\)  in \( R \) both containing \( \textswab{p} \) , \( \textswab{A}\textswab{B} \subseteq \textswab{p} \) implies \( \textswab{A} = \textswab{p} \) or \( \textswab{B} = \textswab{p} \).
			\item[(3)] For left ideals \( \textswab{A}, \textswab{B}\)  in \( R \), \( (\textswab{A}+\textswab{p})(\textswab{B}+\textswab{p}) \subseteq \textswab{p} \)  implies \( \textswab{A} \subseteq \textswab{p} \) or \( \textswab{B} \subseteq \textswab{p} \).
			\item[(4)] For left ideals \( \textswab{A}, \textswab{B}\)  in \( R \), \( \textswab{A}\textswab{B} \subseteq \textswab{p} \) and \( \textswab{p} \subseteq \textswab{A} \) imply \( \textswab{p} = \textswab{A} \) or \( \textswab{B} \subseteq \textswab{p} \).
			\item[(5)] For left ideals \( \textswab{A}, \textswab{B}\)  in \( R \), \( (\textswab{A}+\textswab{p})\textswab{B} \subseteq \textswab{p} \)  implies \( \textswab{A} \subseteq \textswab{p} \) or \( \textswab{B} \subseteq \textswab{p} \).
			\item[(6)] For \( a,b\in R\), \((a+\textswab{p})R(b+\textswab{p})\subseteq \textswab{p}\) implies \( a\in \textswab{p} \) or \(b \in \textswab{p} \).
			\item[(7)]  For a two-sided ideal \( \textswab{A} \subseteq R \) and a left ideal  \( \textswab{B} \subseteq R \)  , \( \textswab{A}\textswab{B} \subseteq \textswab{p} \) and \( \textswab{p}\textswab{B}\subseteq\textswab{p}\) imply \( \textswab{A} \subseteq \textswab{p} \) or \( \textswab{B} \subseteq \textswab{p} \).
			\item[(8)] For \( a,b\in R\), \(aRb\subseteq \textswab{p}\) and \(\textswab{p}Rb\subseteq \textswab{p}\) imply \( a\in \textswab{p} \) or \( b\in   \textswab{p} \).
		\end{itemize}
	\end{proposition}
	\begin{proof}
		The equivalences (2)\(\iff (3) \iff (4) \iff (5) \iff \)(6) are proved in \cite[Proposition 1.1]{van1985weakly}. To prove (1)\(\implies \)(2), suppose  \( \textswab{A}\textswab{B}\subseteq\textswab{p}\) for some left ideals \( \textswab{A},\textswab{B}\) containing \(\textswab{p}\). Then, \( \textswab{p}\textswab{B}\subseteq \textswab{A}\textswab{B}\subseteq\textswab{p}\). If  \( \textswab{p} \) is weakly prime, then \( \textswab{A}\subseteq\textswab{p}\) or \( \textswab{B}\subseteq\textswab{p}\), proving (2). To prove (5)\(\implies \)(1), suppose that  \( \textswab{A}\textswab{B}\subseteq\textswab{p}\) and \( \textswab{p}\textswab{B}\subseteq\textswab{p}\) for some left ideals \( \textswab{A},\textswab{B}\) in \(R\). It follows that 
		\[
			(\textswab{A}+\textswab{p})\textswab{B} \subseteq \textswab{A}\textswab{B}+\textswab{p}\textswab{B} \subseteq \textswab{p}. 	
		\]
		If (5) holds, then either  \( \textswab{A} \subseteq \textswab{p} \) or \( \textswab{B} \subseteq \textswab{p} \), which implies (1). The proof of (1)\(\iff\)(7) is similar to the proof of (1)\(\iff\)(5) in Proposition \ref{prop: characterization of structurally prime ideals}.   The equivalence of (6) and (8) is immediate. 
	\end{proof}
	
	It is evident that  for two-sided ideals, the concept of a weakly prime left ideal coincides with the concept of a structurally prime left ideal. For left ideals that are not two-sided, the following result gives more characterizations of weakly prime left ideals. 
	
	\begin{proposition}\label{prop: characterization of weakly prime left ideals (not two-sided)}
		For any left ideal \( \textswab{p}\) in a ring \( R \) that is not two-sided, the following statements are equivalent:
		\begin{itemize}
			\item[(1)] \( \textswab{p} \) is weakly prime.
			
			\item[(2)]  For a left ideal \( \textswab{A} \),  \( \textswab{p}\textswab{A} \subseteq \textswab{p} \) implies \( \textswab{A} \subseteq \textswab{p} \).
			
			\item[(3)]  For \( b \in R \), \( \textswab{p}Rb \subseteq \textswab{p} \) implies \( b \in \textswab{p} \).
		\end{itemize}
	\end{proposition}
	\begin{proof}
		For a proof, see \cite[Proposition 1.3]{van1985weakly}.
	\end{proof}
	 As a simple consequence of  Proposition \ref{prop: characterization of weakly prime left ideals (not two-sided)}, we see that a proper left ideal \( \textswab{p}\) in \(R\) satisfying \( \textswab{p}R = R\) is necessarily weakly prime. Here is an example of a left ideal which is weakly prime, but not completely prime as a left ideal: 
	 \begin{example}
	 	\label{exam: weakly, not completely}
		The zero ideal in a matrix ring over a division ring is structurally prime, hence weakly prime.   However, the zero ideal  in a matrix ring over a division ring is not completely prime as a left ideal. See Example \ref{exam: structurally, not completely}.  
	 \end{example}
	 A weakly prime left ideal may not be structurally prime, as shown in the following  example. 
	 \begin{example}
	 	\label{exam: weakly, not structurally}
	 	Consider the ring \(R\) and left ideal \(\textswab{p}\) in Example \ref{exam: not structurally, completely}. Since \(\textswab{p}\) is completely prime, it must be weakly prime. Note that \(\textswab{p}\) is not structurally prime. 
	 \end{example}
\end{subsection}
\begin{subsection}{More Characterizations of Primeness} 
	This part gives more characterizations of the reviewed notions of primeness in terms the notion of a left quotient ideal by an element. For a left ideal \(I\) and an element \(a\) in a ring \(R\), the \textit{left quotien ideal} of \(I\) by \(a\) is defined as
	\[
		(I:a) = \left\lbrace x\in R \,|\,  xa\in I\right\rbrace. 
	\]
	It is easy to see that \((I:a)\) is always a left ideal. The largest two-sided ideal in a one-sided ideal \(I\) is denoted by \(\operatorname{id}(I)\). The following proposition provides characterizations of the four notions of primeness for one-sided ideals in terms of \((I:a)\) and \(\operatorname{id}(I)\). The easy proof is left to the reader. 
	\begin{proposition}
		For any proper left ideal \(\textswab{p}\) in a ring \(R\), the following statements hold: 
		\begin{itemize}
			\item[(1)] \(\textswab{p}\) is extremely prime iff \((\textswab{p}:b)\subseteq \textswab{p}\) for all \(b\in R\setminus \textswab{p}\).
			
			\item[(2)]  \(\textswab{p}\) is completely prime iff \((\textswab{p}:b) = \textswab{p}\)  for all \(b\in R\setminus \textswab{p}\) such that \(\textswab{p}\subseteq (\textswab{p}:b)\).
			
			\item[(3)]  \(\textswab{p}\) is structurally prime iff \(\operatorname{id}((\textswab{p}:b)) \subseteq \textswab{p}\)  for all \(b\in R\setminus \textswab{p}\).
			\item[(4)] \(\textswab{p}\) is weakly prime iff \(\operatorname{id}((\textswab{p}:b)) \subseteq \textswab{p}\)  for all \(b\in R\setminus \textswab{p}\) such that \(\textswab{p}R\subseteq (\textswab{p}:b)\).
		\end{itemize}
	\end{proposition}

\end{subsection}


\end{section} 

\begin{section}{Prime Left Ideals in Principal Ideal Domains}\label{sec: primness for PID}
	 
	 In this section, we study the four notions of one-sided primeness, mentioned in the previous chapter, in the context of principal ideal domains. 
	 
\begin{subsection}{Prinicipal Ideal Domains} 
	 For the reader’s convenience, we briefly review some necessary concepts regarding principal ideal domains. Most of the material is standard and can be found in, for example, \cite{cohn2006free}. A domain in which all left ideals and all right ideals are principal is called a \textit{principal ideal domain}, or PID for short. In what follows, let \(R\) be a PID. 
	 
	 An element \(a\in R\) is called a \textit{right }(resp., \textit{left}) \textit{factor} of \(b\in R\) if \(b\in Ra\) (resp., \(b\in aR\)).  A left factor \(a\) of \(b\) is called \textit{proper} if \(a\notin bR\). An element \(a\in R\) is called a  \textit{factor} of \(b\) if \(b = ras \) for some \(r,s              \in R\).  An element \(a\in R\) is called a
	 \textit{right }(resp., \textit{left}) \textit{associate} of \(b\in R\) if \(b = au\) (resp., \(b = ua\)) for some unit \(u\in R\), that is, \(aR = bR\) (resp., \(Ra = Rb\)). 
	 An element \(a\in R\) is called an \textit{associate} of \(b\in R\) if \(a = ubv\) for some units \(u,v\in R\).
	 Elements \(a,b\in R\) are called \textit{similar} iff \(R/Ra\) and \(R/Rb\) are isomorphic as left \(R\)-modules. It is known that the notion of similarity is left-right symmetric, that is, \(a,b\in R\) are similar iff \(R/aR\) and \(R/Rb\) are isomorphic as right \(R\)-modules.  A non-unit \(a\in R\) is said to be \textit{irreducible} (or an \textit{atom}) if a relation \(a = bc\) in \(R\) implies that \(b\) or \(c\) is a unit.	Every PID is a UFD in the following sense: a nonzero element \(a\in R\) can be written as a product of atoms, and furthermore,   if \(a=a_1\cdots a_n = b_1\cdots b_m\) are two factorizations of \(a\) into atoms, then \(m=n\) and there is a permutation \(\pi\in S_n\) such that \(a_i\) is similar to \(b_{\pi(i)}\) for all \(i\).

	  A relation \(ax = by\) in \(R\) is called a \textit{left} (resp., \textit{right}) \textit{comaximal relation} if \(aR+bR = R\)  (resp., \(Rx+Ry = R\)). A comaximal relation is a relation that is both left and right comaximal. 	 The following result describes  similar elements in terms of comaximal relations  in PIDs (we include a proof, as the author was unable to locate a suitable reference in the existing literature):
	  \begin{proposition}\label{prop: similarity in PID}
	  	In any PID \(R\), elements \(a,b\in R\) are similar iff they satisfy a comaximal relation \(ax=yb\) in \(R\).
	  \end{proposition}
	  \begin{proof}
	  	First, let \(a,b\in R\) be similar, i.e., \(R/Ra\) and \(R/Rb\) are isomorphic as left \(R\)-modules. It follows from \cite[Proposition 1.3.6]{cohn2006free} that there exists \(x\in R\) such that  \(Rb + Rx = R \) and  \((Rb:x) = Ra\). We have \(ax = yb\) for some \(y\in R\). To prove that \(ax = yb\) is a comaximal relation, we need to show that \(aR+yR=R\). We have \(aR+yR = dR\) for some \(d\in R\), implying \(a=da_1, y=dy_1\) for some \(a_1,d_1\in R\). The relation \(ax = yb\) simplifies to \(a_1x = y_1b\). Since \(a_1\in (Rb:x) = Ra\), \(d\) must be a unit because \(a_1\in Ra\) and \(a\in Ra_1\). This completes the proof of the forward direction. 
	  	
	  	Conversely, let  \(ax=yb\) be a comaximal relation in \(R\). It is easy to verify that the map 
	  	\[
	  	r+Ra\mapsto rx+Rb 
	  	\]
	  	is an isomorphism between
	  	the left \(R\)-modules \(R/Ra\) and \(R/Rb\), i.e.,  \(a,b\in R\) are similar, and we are done. 
	  \end{proof}	
	 
	 A nonzero element \(a\in R\) is called \textit{invariant} if \(Ra = aR\).  It is easy to see that a nonzero element \(a\in R\) is invariant iff \(aR\subseteq Ra\) iff \(Ra\subseteq aR\). Every nonzero ideal of \(R\) is of the form \(Ra=aR\) for some invariant element \(a\in R\), which is unique up to (right or left) associates. For every \(a\in R\) and ideal \(I\) in \(R\), \(I\subseteq Ra\) iff \(I\subseteq aR\) (proof: if \(Rb=bR\subseteq Ra\implies \exists r,\, b=ra \implies\exists s,\, rb = bs=ras\implies b\in aR\implies I\subseteq aR\)). An element \(a\in R\) is called \textit{bounded} if \(Ra\) contains a nonzero ideal, or equivalently,  \(aR\) contains a nonzero ideal. For a bounded element \(a\), a bound \(a^*\) of \(a\) is defined to be a generator of the largest ideal contained in \(Ra\). The element \(a^*\) is unique up to associates, called the \textit{bound} of \(a\). Note that \(Ra^*\) is also the largest ideal contained in \(aR\). If \(a^*\) is the bound of \(a\), then
	 \[
	 	Ra^* = a^*R = Ann(\frac{R}{Ra}) = Ann(\frac{R}{aR}).
	  \] 
	 In particular, similar elements have the same bound. It is easy to show that an element \(a\in R\) is bounded iff \(ras\) is invariant for some nonzero elements \(r,s\in R\).

\end{subsection}

	\begin{subsection}{Extremely Prime Left Ideals in PIDs}
	In this part, we give a description of extremely prime left ideals of a PID. In what follows, let \(R\) be a PID. Since \( R \) is a domain, the zero ideal is extremely prime. The following proposition describes nonzero extremely prime left ideals of \(R\).
	\begin{proposition}\label{prop: extremeply prime ideal of PID }
		A nonzero left ideal \( Ra \) in \( R \) is extremely prime iff \(a\) is irreducible and invariant.
	\end{proposition}
	\begin{proof}
		
		First, suppose that \(  Ra \) is a nonzero extremely prime left ideal of  \( R \). Using the fact that \(R\) is (left and right) noetherian, we can write \( a = a_1\cdots a_n\), where \(a_1,\dots,a_n\in R\) are irreducible. Since \(Ra\) is extremely prime, \(a_i\in Ra\) for some \(i\).  Since \(a_i\) is irreducible and \(a\) is a non-unit,  \( a \) must be irreducible. 
		Fix \(c\in R\setminus Ra\).  Then \((Ra : c) = Rb\) for some nonzero element \( b\in R\). Since \(bc\in Ra\) and \(c\notin Ra\), we must have \( b \in Ra \). It is easy to verify that the map 
		\[
		x+Rb\mapsto xc+Ra 
		\]
		is an isomorphism between
		the left \(R\)-modules \(R/Ra\) and \(R/Rb\), i.e., \(a,b\) are similar. In particular, \(b\) is irreducible, and since \(b\in Ra\), we conclude that \((Ra : c)=Ra\). This proves that  \(aR\subseteq Ra\), implying that \(a\) is invariant. This completes the proof of the forward direction.
		
		Conversely, suppose that \(Ra\) is a nonzero left ideal such that \(aR = Ra\) and  \(a\) is irreducible. Let \(bc\in Ra\) and \(c\notin Ra\). We need to show that \(b\in Ra\). Since \(a\) is irreducible and \(c\notin aR\), we see that \(cR+aR=R\). Multiplying by \(b\), we obtain \(bcR+baR=bR\). Then
		\[
			b\in bR = bcR+baR \subseteq RaR + RaR = Ra.
		\] 
		This proves that 
		\(Ra\) is extremely prime, completing the proof of the reverse direction. 
	\end{proof}
	As a consequence of the proposition, we note that any extremely prime left ideal of a PID is a completely prime two-sided ideal, and vice versa. 
	
	The structure of invariant elements in skew polynomial rings is well understood, thanks to the work of Cauchon \cite{cauchon2006ideaux}, Carcanague \cite{carcanague1971ideaux}, Lam and Leroy \cite{lam1988algebraic}, among others. We give some examples using their results.  
	\begin{example}
		Let \(D[t]\) denote the ring of polynomials in one variable over a division ring \(D\). It is well-known that \(D[t]\) is a PID. It is easy to see that a monic polynomial in \(D[t]\) is invariant iff it is central. By Proposition \ref{prop: extremeply prime ideal of PID }, a nonzero left ideal \(I\) in \(D[t]\) is extremely prime iff \(I=D[t]f(t)\), where \(f(t)\) is a central irreducible polynomial. In the case where \(D = \mathbb{H}\) is the division ring of quaternions over \(\mathbb{R}\), a nonzero left ideal \(I\) in \(\mathbb{H}[t]\)  is extremely prime iff \(I=\mathbb{H}[t](t-a)\) for some \(a\in\mathbb{R}\). This is a consequence of the fundamental theorem of algebra for quaternions \cite[Thoerem 16.14]{lam2001afirst}. 
	\end{example}
	
	\begin{example}
		Let \(D[t;\sigma]\) denote the ring of skew polynomials associated to an automoprhism \(\sigma\) of a division ring \(D\). It is well-known that \(D[t;\sigma]\) is a PID. If \(\sigma\) is of infinite inner order, then every  invariant elements of \(D[t;\sigma]\) is of the form \(at^m\) for some \(a\in D\setminus \{0\}\) and \(m\geq 1\). It follows that \(D[t]t\) is
		the only nonzero extremely prime left ideal of \(D[t;\sigma]\) provided that \(\sigma\) is of infinite inner order.  
	\end{example}
	
	\end{subsection}
	\begin{subsection}{Completely Prime Left Ideals in PIDs }
		We begin with a definition that is needed to describe completely prime one-sided ideals of PIDs. A nonzero element \(a\) in a ring \(R\) is said to be \textit{c-reducible} if \(a\) can be written  as \(a = bb' = c'c\) for some nonunits \(b,b',c',c\in R\) such that \(b\) is similar to \(c\). An element that is not c-reducible will be called \textit{c-irreducible}. 
		
		As the zero ideal is completely prime as a left ideal in any domain, we only need to deal with nonzero left ideals. The following result gives a description of nonzero completely prime left ideals of PIDs. 
		\begin{proposition}\label{prop: completely prime ideal of PID}
			A nonzero left ideal \(Ra\) of a PID \(R \) is completely prime iff  \( a \) is  c-irreducible.
		\end{proposition}
		\begin{proof}
			We prove the contrapositive of the statement, that is, \( a \) is  c-reducible iff \(Ra\) is not completely prime. First,
			assume that \(a\) is c-reducible. \(a = bb' = c'c\) for some nonunits \(b,b',c',c\in R\) such that \(b\) is similar to \(c\).   By Proposition \ref{prop: similarity in PID},  there exists a comaximal relation \(c x = y b\) in \(R\). We have \(c(xb') = ybb'\in Ra\). Moreover, \((Ra)xb'=Rc'cxb'=Rc'ybb'\subseteq Ra\). Since \(c\notin Ra\) and \(xb'\notin Ra\), we conclude that \(Ra\) is not completely prime. This completes the proof of the forward direction. 
			
			Conversely, suppose that \(Ra\) is not completely prime. Then, there exist \(b,c\in R\setminus Ra\) such that \(bc\in Ra\) and \((Ra)c\subseteq Ra\). It follows that \(b\in (Ra:c)\) and \(a\in (Ra:c)\). We have \((Ra:c) = Rb_0\) for some \(b_0\in R\). Therefore, \(a = b'b_0 \)  for some \(b'\in R\). Since \(Ra+Rc = Rc_1\) for some \(c_1\in R\), we see that \(a=c_2c_1\) and  \(c=x c_1\) for some \(c_2,x\in R \) such that \(Rx+Rc_2=R\). It follows from \((Ra:c) = Rb_0\) that \(b_0c = y_0 a \) for some \(y_0\in R\). Therefore,  \(b_0xc_1 = b_0c =  y a = y c_2c_1 \), implying \(b_0x=y_0c_2\). We have \(b_0 = d b_1\) and \(y_0 = d y\) for some \(d,b_0,y\) such that \(b_1R+yR=R\). Since \(b_1x = y c_2 \), \(Rx+Rc_2=R\), and \(b_1R+yR=R\), it follows from Proposition \ref{prop: similarity in PID} that \(b_1,c_2\) are similar. Since \(a = (b'd)b_1 = c_2c_1\), we conclude that \(a\) is c-reducible. Note that \(b'd,b_1,c_2,c_1\) are nonunits, thanks to the assumption \(b,c\in R\setminus Ra\). 
		\end{proof}
		Since the notion of c-irreducibility is left-right symmetric, it follows from Proposition \ref{prop: completely prime ideal of PID} that \(Ra\) is a completely prime left ideal iff \(aR\) is a completely prime right ideal. 
		
		It is clear that every irreducible element is c-irreducible. The converse does not hold, as shown in the following example:
		\begin{example}
			Let \(\frac{d}{dx}\) denote the derivative on the ring \(\mathbb{R}(x)\) of rational functions over  \(\mathbb{R}\). Consider the skew polynomial ring \(R = \mathbb{R}(x)[t;1,\frac{d}{dx}]\). It is known that \(R\) is a PID. Set \(f = (t+x)t\in R\). If \(f = (t+a(x))(t+b(x))\) is a factorization of \(f\) in \(R\), then \(b(x)\) must satisfy the Bernoulli equation \(y'+xy=y^2\) whose only rational solution is \(y=0\). It follows that the element \(f\) has a unique factorization \(f = (t+x)t\) into monic nonconstant polynomials. It is easy to verify that \(t+x\) and \(t\) are not similar in \(R\).  Therefore, \(f\) is  c-irreducible, but not irreducible. For another example of a c-irreducible element that is not irreducible, see \cite[Proposition 4.5]{alon2022completely}. 
		\end{example}
		
		 Under certain assumptions, a c-irreducible element must be irreducible: 
		
		\begin{proposition}\label{prop: completely prime with bounded factor}
			Let \(a\in R\) be an element in a PID. Assume that \(a\) has a non-unit bounded factor. Then \(a\) is c-irreducible iff it is irreducible.
		\end{proposition}
		\begin{proof}
			We only need to prove the forward direction. Let \(a = p_1\cdots p_n\in R\) be a factorization of a c-irreducible element \(a\) into irreducible elements. Since \(a\) has a bounded factor, some \(p_i\) is bounded. Let \(p_{i_1},\dots,p_{i_m}\) be all the factors that are similar to \(p_i\). It follows from Lemma 6.2.5 in  \cite[p. 344]{cohn2006free}  that 
			\[
				a = (p_{i_1}\cdots p_{i_m}) b = c b',
			\]
			where \((p_{i_1}\cdots p_{i_m})\) and \(b'\) are similar (see also the discussion following the lemma in Cohn's book \cite[p. 344]{cohn2006free}). Since \(a\) is c-irreducible, \(m=1\)  and \(a\) must be similar to \(p_i\), implying that \(a\) is irreducible.
		\end{proof} 
	Following Feller \cite{feller1961factorization}, we call a PID \textit{bounded} if every non-unit element in the PID is bounded. It is easy to show that a PID \(R\) is bounded iff every atom in \(R\) is bounded. Consequently, a PID is bounded iff it is not (left, or equivalently, right) primitive. For a  detailed study of the notion of ``boundedness" for arbitrary rings, see \cite[Chapter 1]{goodearl2004introduction}. As a direct consequence of Proposition \ref{prop: completely prime with bounded factor}, we present the following result. 
	\begin{corollary}\label{prop: bounded PID}
		Let \(R\) be a bounded PID. Then, \(a\in R\) is c-irreducible iff it is irreducible. In other words, in any bounded PID, a nonzero left ideal is completely prime as a left ideal iff it is maximal as a left ideal.
	\end{corollary}
		 We now give some examples of bounded PIDs.   
		\begin{example}
			Let \(D\) be a division ring. It is known that  \(D[t]\) is a bounded PID if every square matrix over \(D\) is algebraic over the center of \(D\) \cite[Theorem 17.2.]{goodearl2004introduction}. In particular, if \(D\) is centrally finite, i.e., is of finite dimension over its center, then \(D[t]\) is a bounded PID. 
		\end{example} 
		\begin{example}
			Let \(\delta\) be a derivation on a field \(k\) of characteristic \(p>0\). If \(\delta\) is nilpotent, then \(R = k[t;\delta]\) is a bounded PID \cite[Exercise 9D]{goodearl2004introduction}. It can also be proved using the results in the next example.   
		\end{example}  
		\begin{example}
			For the terminology used in this example, we refer the reader to \cite{leroy1995pseudo}. A skew polynomial ring \(D[t;\sigma,\delta]\) over a division ring \(D\) is left primitive iff \(D[t;\sigma,\delta]\) satisfies one of the following conditions: (1) \(D[t;\sigma,\delta]\) is simple; (2)  \(D[t;\sigma,\delta]\) is not simple and no positive power of \(\sigma\) is an inner automorphism; (3)  \(D[t;\sigma,\delta]\) is not simple, a positive power of \(\sigma\) is an inner automorphism, and there exist \(n\geq 1\) and a square matrix \(A\in M_n(D)\) such that \(A\) is not algebraic over the center of \(D\) \cite[Theorem 3.26]{leroy1995pseudo}. Since boundedness and nonprimitiveness are equivalent for PIDs, we see that \(D[t;\sigma,\delta]\) is a bounded PID iff \(D[t;\sigma,\delta]\) is not simple, some positive power of \(\sigma\) is an inner automorphism, and every square matrix over \(D\) is algebraic over the center of \(D\). As an example, if \(\sigma:k\to k\) is an automorphism of finite order over a field \(k\), then \(k[t;\sigma]\) is a bounded PID. In particular,  a left ideal of \(k[t;\sigma]\) is completely prime iff it is a maximal left ideal. For another proof of the last statement, see \cite[Theorem 1.1]{alon2022completely}. 
		\end{example} 
		\begin{example}
			In this example, we assume that the reader is familiar with the work of  Lam and Leroy \cite{lam2004wedderburn}. Let \( D[t;\s,\delta] \) be a skew polynomial ring over a division ring. Suppose that the \((\sigma,\delta)\)-conjugacy class of an element \(a\in D\) is \((\sigma,\delta)\)-algebraic. Then, any polynomial \(f(t)\) of degree at least \(2\) that has \(a\) as a right root is necessarily c-reducible.  This is because, by Corollary 6.3 in \cite{lam2004wedderburn}, \(f(t)\) must have a left root \(b\) that is \((\sigma,\delta)\)-conjugate  to \(a\). Note that \(t-a\), a proper left factor of \(f(t)\), and \(t-b\),  a proper right factor of \(f(t)\), are similar. In particular, for such an element \(f\), the left ideal \(D[t]f(t)\) is never completely prime as a left ideal. 
		\end{example} 
		
	\end{subsection}
	\begin{subsection}{Structurally Prime Left Ideals in PIDs}		
		To give a characterization of structurally prime left ideals in a PID, we need to review some concepts. Recall that an invariant non-unit \(a\in R\) is called \textit{Inv-irreducible} (or an \textit{Inv-atom}) if \(a = bc\), where \(b,c\in R\) are invariant, implies that either \(b\) or  \(c\) is a unit. It is easy to see that \(a\in R\) is Inv-irreducible iff \(Ra\) is  a maximal ideal \cite[Theorem 6.2.4 (i)]{cohn2006free}. It is also known that an element \(a\in R\)  in a PID is Inv-irreducible iff \(a\) is the bound of a bounded atom \(b\), in which case \(a\) is a product of similar bounded atoms \cite[Theorem 6.2.4 (ii), (iii)]{cohn2006free}.  A nonzero non-unit is called \textit{left totally unbounded} if it has no non-unit bounded left factors. 
		\begin{proposition}\label{prop: structurally prime ideal of PID}
			A nonzero left ideal \(Ra\) of a PID \(R \) is  structurally prime iff  \( a \) is  left totally unbounded or a factor of an Inv-atom.
		\end{proposition}
		\begin{proof}
			To prove the forward direction, suppose that \(Ra\) is structurally prime. If \(a\) is left totally bounded, there is nothing to prove. Assume that \(a\) is not left totally bounded. We have \(a = pb\), where \(p\) is a bounded irreducible left factor of \(a\). Let \(c = rp\) be the bound of \(p\). Then \(cRb=R(rp)b\subseteq Ra\). Since \(Ra\) is structurally prime and \(b\notin Ra\), we see that \(c\in Ra\), i.e., \(a\) is a factor of the bound of an atom, which is Inv-reducible by \cite[Theorem 6.2.4 (iii)]{cohn2006free}. This completes the proof of the forward direction.

			Conversely, suppose that \( a \) is  left totally unbounded or a factor of an Inv-irreducible element. Suppose \(bRc\subseteq Ra\) for some nonzero \(b,c\in R\). To prove that \(Ra\) is structurally prime, we need to show that \(b\in Ra\) or \(c\in Ra\). We have \(Ra+Rc=Rx\) for some \(x\in R\). It follows that \(a=a_1x, c=c_1x\) for some \(a_1,c_1\in R\) such that \(Ra_1+Rc_1=R\). Multiplying by \(b\), we obtain 
			\[	
			bR =  bRa_1 + bRc_1 \subseteq Ra_1.
			\]			
			It follows that \(a_1\) is bounded. If \(a\) is left totally bounded, \(a_1\) must be a unit, and consequently, \(c\in Ra\). We may, therefore, assume that \( a \) is  a factor of an Inv-atom \(p\) and \(a_1\) is a non-unit. Let \(d\) be the bound of \(a_1\). Note that \(b\in Rd\). Since \(a\in a_1R\), \(p\) belongs to \(Rd\). Since \(p\) is Inv-irreducible, we see that \(Rd = Rp\), hence \(b\in Rd = Rp\subseteq Ra\). This completes the proof.  
		\end{proof}
		As a direct consequence of Proposition \ref{prop: characterization of structurally prime ideals}, we see that a nonzero two-sided ideal of a PID \(R\) is prime as a two-sided ideal iff it is a maximal two-sided ideal. 
		\begin{example}\label{exam: structurally prime ideals in D[x]}
			Let \(D\) be a division ring with center \(k\). It is easy to see that a  monic polynomial \(f\in D[x]\) is invariant iff \(f\in k[x]\). Moreover, a nonconstant polynomial \(f\in k[x]\) is Inv-irreducible iff \(f\) is irreducible as an element of \(k[x]\). By Proposition \ref{prop: characterization of structurally prime ideals}, a nonzero left ideal \(D[x] f(x) \neq D[x]\) is structurally prime, iff \(f(x)\) is not a factor of a nonzero  polynomial in \(k[x]\), or \(f(x)\) is a factor of an irreducible polynomial in \(k[x]\). In the case where \(D = \mathbb{H}\) is the division ring of quaternions, the center of \( \mathbb{H}[x]\) is \( \mathbb{R}[x]\),  and  a nonconstant polynomial \(f\in \mathbb{R}[x]\) is Inv-irreducible iff it has degree \(1\) or \(f(t) = t^2+rt+s\) , where \(r^2<4s\). 
		\end{example}
		
		In general rings, a completely prime ideal may not be structurally prime (see Example \ref{exam: not structurally, completely}). This is, however, true for PIDs:  
		\begin{proposition}\label{prop: structurally prime implies completely prim in PID}
			A completely prime left ideal in any PID is structurally prime 
		\end{proposition}
		\begin{proof}
			Let \(Ra\) be a nonzero completely prime ideal in a PID \(R\). By Proposition \ref{prop: completely prime with bounded factor}, \(a\) is either irreducible or totally unbounded. It follows from Proposition \ref{prop: structurally prime ideal of PID} that  \(Ra\) is structurally prime.   	
		\end{proof}
		The converse of Proposition \ref{prop: structurally prime implies completely prim in PID} does not hold, as demonstrated by the following example. 
		\begin{example}\label{exam: structurally not completely in a PID}
			The left ideal \(\textswab{p} = \mathbb{H}[t](t^2+1)\) is structurally prime since the invariant element \(t^2+1\) is Inv-irreducible.  However, \(\textswab{p}\) is not completely prime because \((t-i)(t+i)\in \textswab{p}\) and \(\textswab{p}(t+i)\subseteq \textswab{p}\), but \(t-i,t+i\notin \textswab{p}\). For another example, see \cite[Example 1]{alon2022completely}.
		\end{example}

	\end{subsection}
	\begin{subsection}{Weakly Prime Left Ideals in PIDs}
		In general, a two-sided ideal is weakly prime as a left ideal iff it is structurally prime as a left ideal. Therefore, we restrict our attention to left ideals that are not two-sided.  
		\begin{proposition}\label{prop: weakly prime ideal of PID }
			Let \(Ra\) be a left ideal in a PID \(R\) that is not two-sided. Then \(Ra\)  is  weakly prime iff  \( a \) has no non-unit invariant factors. In other words, a left ideal in a PID that is not two-sided is weakly prime iff it is not contained in a proper two-sided ideal. 
		\end{proposition}
		\begin{proof}
			We prove the contrapositive of the statement, that is,  \( a \) has a non-unit invariant factor  iff \(Ra\) is not weakly prime. First, assume that  \( a \) has a non-unit invariant factor \(b\). It follows that \(a = bx\) for some \(x\in R\), implying \(aRx \subseteq bRx = Rbx = Ra\). Since \(x\notin Ra\), it follows from Proposition \ref{prop: characterization of weakly prime left ideals (not two-sided)} (see Part (3)) that \(a\) is not weakly prime. 
			
			Conversely, suppose that \(Ra\) is not weakly prime. Since \(Ra\) is not two-sided, by Part (3) of Proposition \ref{prop: characterization of weakly prime left ideals (not two-sided)}, there exists \(b\in R\setminus Ra\) such that \(aRb\in Ra\). We can write \(a = cx\) and \(b = dx\) for some \(b,d,x\in R\) such that \(Rc + Rd = R\). Note that \(aRd\subseteq Rc\). Since \(b\notin Ra\), \(c\) must be a non-unit. Furthermore, we have 
			\[
				aR = aRc + aRd \subseteq Rc.
			\]
			It follows that \(c\) is a bounded element whose bound is a  factor of \(a\),  completing the proof. 
		\end{proof}
		We end the article with the following example.  
		\begin{example}\label{exam: weakly prime ideals in D[x]}
			Let \(D\) be a division ring with center \(k\). Any left ideal of \(D[x]\) that is not two-sided is of the form \(D[x]f(x)\), where \(f(x)\) is a monic nonconstant polynomial not contained in \(k[x]\). For such a left ideal, \(D[x]f(x)\) is weakly prime iff \(f(x)\) does not have nonunit factos belonging to \(k[x]\).  In the case where \(D = \mathbb{H}\) is the division ring of quaternions, the reader can check that the left ideal \( \mathbb{H}[x](x-j)(x-i)\) is weakly prime, but not structurally prime.  
		\end{example}
	\end{subsection}

\end{section}


\bibliographystyle{plain}
\bibliography{PObiblan}

 \end{document}